\documentclass[11pt, one sided]{amsart}
\RequirePackage[colorlinks,citecolor=blue,urlcolor=blue]{hyperref}

\usepackage{latexsym,amssymb,amsmath,amsfonts,amsthm,mathtools}
\usepackage{amsthm,amsmath}
\usepackage{float}
\usepackage{enumerate}
\usepackage[margin=3cm]{geometry}
\usepackage{bbm}
\usepackage{subfigure}
\usepackage{graphicx}
\usepackage[toc,page]{appendix}
\usepackage{multirow}
\usepackage[all]{xy}
\usepackage{caption}
 \usepackage{geometry}                
\usepackage{graphicx}
\usepackage{amssymb}
\usepackage{epstopdf}
\usepackage{color}
\usepackage{bbm, dsfont}
\usepackage{hyperref}
\usepackage[utf8]{inputenc}
\usepackage{comment}

\usepackage{amssymb,amsthm,amsmath,amssymb,wrapfig,dsfont}
\usepackage[dvipsnames]{xcolor}
\definecolor{myred}{RGB}{251,154,133}
\definecolor{myblue}{RGB}{153,206,227}
\definecolor{mylightblue}{RGB}{0, 150, 255}
\definecolor{mygreen}{RGB}{32, 210, 64}
\definecolor{mygray}{RGB}{220, 220, 220}

\usepackage{tikz}
\usetikzlibrary{decorations.pathmorphing}
\tikzset{snake it/.style={decorate, decoration=snake}}
\usetikzlibrary{shapes.geometric,positioning,decorations.pathreplacing}

\usepackage{pgfplots}

\newtheorem{theorem}{Theorem}

\newtheorem{lemma}{Lemma}[section]

\newtheorem{corollary}{Corollary}

\def\beq{ \begin{equation} }
\def\eeq{ \end{equation} }

\def\square{\vcenter{\vbox{\hrule height .4pt
  \hbox{\vrule width .4pt height 5pt \kern 5pt
        \vrule width .4pt} \hrule height .4pt}}}

\newcommand{\BR}{{\mathbb{R}}}

\newcommand{\bae}{\begin{equation}\begin{aligned}}
\newcommand{\eae}{\end{aligned}\end{equation}}

\DeclareFontFamily{OML}{rsfs}{\skewchar\font'177}
\DeclareFontShape{OML}{rsfs}{m}{n}{ <5> <6> rsfs5 <7> <8> <9>
rsfs7 <10> <10.95> <12> <14.4> <17.28> <20.74> <24.88> rsfs10 }{}
\DeclareMathAlphabet{\mathfs}{OML}{rsfs}{m}{n}

\DeclarePairedDelimiter{\ceil}{\lceil}{\rceil}
\DeclarePairedDelimiter{\floor}{\lfloor}{\rfloor}

\newcommand{\rdb}{\rho_{\text{DB}}}

\newcommand{\note}[1]{{\color{red}{ \bf{ [Note: #1]}}}}

\begin{document}

\title[]{Discrete $\ell^{1}$ Double Bubble solution is at most\\ ceiling +2 of the continuous solution}

\author{Parker Duncan}
\address[Parker Duncan]{Faculty of Industrial Engineering and Management, Technion - Israel Institute of Technology and Department of mathematics, Texas A\&M University}
\email{parkeraduncan@tamu.edu}

\author{Rory O'Dwyer}
\address[Rory O'Dwyer]{Department of Physics, Stanford University}
\email{rodwyer@stanford.edu}

\author{Eviatar B. Procaccia}
\address[Eviatar B. Procaccia\footnote{Research partially supported by NSF grant 1812009 and BSF grant 2018330.}]{Faculty of Industrial Engineering and Management, Technion - Israel Institute of Technology}
\urladdr{https://procaccia.net.technion.ac.il}
\email{eviatarp@technion.ac.il}

\maketitle
\begin{abstract}
In this paper we show that the solution of the discrete Double Bubble problem over $\mathbb{Z}^2$ is at most the ceiling function  plus two of the continuous solution to the Double Bubble problem, with respect to the $\ell^1$ norm, found in  \cite{morgan1998wulff} and \cite{duncan2020elementary}.
\end{abstract}





\section{Introduction}
The Double Bubble problem asks the following: given two volumes, what are the two shapes admitting these volumes with the smallest perimeter, where the perimeter of the joint boundary is counted once. In the Euclidean setting in $\BR^2$ and $\BR^3$ the solution was established in \cite{foisy1993standard} and \cite{doublebubbleconj}. The Euclidean solution is given by three spherical caps whose tangents meet at an angle of 120 degrees. The same solution was shown to be valid under the Gaussian measure in \cite{milman2018gaussian}.

The Double Bubble problem is a generalization of the Isoperimetric problem. Discrete versions of the Isoperimetric problem arise naturally in Probability Theory \cite{alexander1990wulff, biskup2015isoperimetry, bodineau2000rigorous, cerf2006wulff, procaccia2012concentration, procaccia2016quenched}. Alonso and Cerf \cite{alonso1996three} established the solution of the Isoperimetric problem on the three dimensional lattice, proving it is very close to the continuous solution. In this paper we prove that the solution of the perimeter of the discrete Double bubble solution on the square lattice is at most the ceiling plus two of the continuous solution, where the perimeter is taken with respect to the $\ell^1$ norm. While finishing this paper, Friedrich, G{\'o}rny, and Stefanelli \cite{friedrich2021double} have uploaded a paper with mutually exclusive but very related results — mainly showing that different solutions of the discrete Double Bubble problem are close to each other in a geometric sense.

\subsection{Notations and results}

For any Lebesgue-measurable set $A\subset\mathbb{R}^2$, let $\mu(A)$ be its Lebesgue measure.  For a simple curve $\lambda:[a,b]\rightarrow\mathbb{R}^2$, not necessarily satisfying $\lambda(a)=\lambda(b)$, where $\lambda(t)=(x(t),y(t))$, define its $\ell^1$ length by 
$$\rho(\lambda)=\sup_{N\geq1}\sup_{a\leq t_1\leq...\leq t_N\leq b}\sum_{i=1}^N\big(\left|x(t_{i+1})-x(t_i)\big|+\big|y(t_{i+1})-y(t_i)\big|\right).$$  
If we wish to measure only a portion of the curve $\lambda$, it will be denoted $\rho(\lambda
([t,t']))$, where $[t,t']\subset[a,b]$.  For simplicity we assume that $[a,b]=[0,1]$ unless otherwise stated. 

We say that two curves $\lambda,\lambda':[0,1]\rightarrow\mathbb{R}^2$ intersect nontrivially if there are intervals $[s,s'],[t,t']\subset[0,1]$ such that $\lambda([s,s'])=\lambda'([t,t'])$, then their nontrivial intersection can be written as the union of curves $\lambda_i$ such that $\lambda_i([0,1])=\lambda([s_i,s_{i+1}])\rightarrow\mathbb{R}^2$ for some intervals $[s_i,s_{i+1}]$, and we define the length of the nontrivial intersection to be $\rho(\lambda\cap\lambda'):=\sum_i\rho(\lambda_i)$.  \\

Here we are interested in the double bubble perimeter of two simply connected open sets $A,B\subset\mathbb{R}^2$ where the boundary of $A$, $\partial A$, is a closed, simple, rectifiable curve (similarly for $B$), and the intersection of the boundaries of $A$ and $B$ is a union of disjoint, rectifiable curves. Define $\gamma_\alpha$ to be the collection of such sets $(A,B)$ satisfying $\frac{\mu(B)}{\mu(A)}=\alpha$. For any $\alpha\in(0,1]$ and $(A,B)\in\gamma_\alpha$, the double bubble perimeter of $(A,B)$ is defined as $$\rho_{\text{DB}}(A,B)=\rho(\lambda)+\rho(\lambda')-\rho(\lambda\cap\lambda'),$$ where $\partial A=\lambda([0,1])$, and $\partial B=\lambda'([0,1])$.  We will also use the notation $\rho(\lambda)=\rho(\partial A)$. For any $\alpha\in(0,1]$ and $X,Y\in\BR$ satisfying $\frac{Y}{X}=\alpha$, let
$$
\rho_{cont}(X,Y)=\inf\{\rho_{DB}(A,B):(A,B)\in\gamma_\alpha,\mu(A)=X,\mu(B)=Y\}
.$$
It has been proved in \cite{duncan2020elementary,morgan1998wulff} that

\begin{equation}
\rho_{cont}(X,Y)=\left\{\begin{array}{ll}
4\sqrt{X+Y}+2\sqrt{Y}& \text{ for } 0<\alpha\le\alpha_0\\
4\sqrt{X}+2\sqrt{2Y}&\text{ for } \alpha_0\le \alpha\le \frac{1}{2}\\
  2\sqrt{6(X+Y)} & \text{ for } \frac{1}{2}\le\alpha\le 1,
\end{array}\right.
\end{equation}
where $\alpha_0=\frac{688-480\sqrt{2}}{49}$.

The minimizing double bubble shapes for $\alpha=\frac{Y}{X}$ in the continuous case are given by:

\begin{figure}[H]
\begin{tikzpicture}[scale=0.7]
\draw[blue,  thin] (0,4) to (4,4);
\draw[blue,  thin] (0,0) to (4,0);
\draw[blue,  thin] (0,0) to (0,4);
\draw[blue,  thin] (4,0) to (4,4);
\draw[blue, thin] (3,3) to (3,4);
\draw[blue, thin] (3,3) to (4,3);
\draw (2,0.35) node{$\sqrt{X+Y}$};
\draw (1.05,1.8) node{$\sqrt{X+Y}$};
\draw (2.5,3.5) node{$\sqrt{Y}$};
\draw (3.4,2.65) node{$\sqrt{Y}$};

\draw (2,-0.5) node{$\alpha\in(0,\frac{688-480\sqrt{2}}{49}]$};
\draw[blue,  thin] (5,4) to (9,4);
\draw[blue,  thin] (5,0) to (9,0);
\draw[blue,  thin] (5,0) to (5,4);
\draw[blue,  thin] (9,0) to (9,4);
\draw (7,0.35) node{$\sqrt{X}$};
\draw (5.5,1.8) node{$\sqrt{X}$};
\draw (7,-0.5) node{$\alpha\in [\frac{688-480\sqrt{2}}{49},0.5]$};
\draw[blue,  thin] (9,1) to (10.5,1);
\draw[blue,  thin] (9,3) to (10.5,3);
\draw[blue,  thin] (10.5,1) to (10.5,3);
\draw (9.88,1.8) node{$\sqrt{2Y}$};
\draw (9.80,3.5) node{$\frac{\sqrt{2Y}}{2}$};

\draw[blue,  thin] (11,4) to (14.5,4);
\draw[blue,  thin] (11,0) to (14.5,0);
\draw[blue,  thin] (11,0) to (11,4);
\draw[blue,  thin] (15,0) to (15,4);
\draw[blue,  thin] (14.5,4) to (18,4);
\draw[blue,  thin] (14.5,0) to (18,0);
\draw[blue,  thin] (18,4) to (18,0);
\draw (12.15,1.8) node{$\sqrt{\frac{2(X+Y)}{3}}$};
\draw (12.8,3.44) node{$X\sqrt{\frac{3}{2(X+Y)}}$};
\draw (16.5,3.44) node{$Y\sqrt{\frac{3}{2(X+Y)}}$};
\draw (14,-0.5) node{$\alpha\in[0.5,1]$};

\end{tikzpicture}
\caption{\label{fig:Figure4}}
\end{figure}


In this paper we want to constrain ourselves to the discrete problem.  That is, the boundaries of the sets in which we are interested should always have at least one coordinate that is an integer, i.e. if $a\in\partial A$, and $a=(a_x,a_y)\in\mathbb{R}^2$, then either
 $a_{x}\in\mathbb{Z}$ or $a_{y}\in\mathbb{Z}$, similarly for $\partial B$. This constrains the figures so that they lie on the square lattice. We denote this collection of sets lying on the square lattice by $\mathfs{L}$.\\



For $n,m\in\mathbb{Z}$ such that $\alpha=\frac{m}{n}$, define: $$\gamma_{n,m}=\{(A,B)\in\gamma_\alpha:A,B\in\mathfs{L}\text{ with }\mu(A)=n,\mu(B)=m\}.$$

Finally because for any $(A,B)\in\gamma_{n,m}$, $\rdb(A,B)\in\mathbb{Z}^+$, $\Gamma_{n,m}=\text{arginf}\{\rdb(A,B):(A,B)\in\gamma_{n,m})\}$ exists. The focus of this paper will to be to study this set. As a final note, in order to emphasize the role of the continuous double bubble solution in our analysis, we will always use $\rho_{cont}(X,Y)$ for the continuous minimum for the volumes X and Y and use the notation $\rho_{DB}(\Gamma_{n,m})$, for the discrete case. 


In this paper we will obtain the following theorems:

\begin{theorem}
For $0<\alpha\le1/2$, $m,n\in\mathbb{Z}^+$ such that $m/n=\alpha$, 
$$\lceil\rho_{cont}(n,m)\rceil\le\rdb (\Gamma_{n,m})\le\lceil\rho_{cont}(n,m)\rceil+2.$$
\label{theorem:Thm1}
\end{theorem}

For the case of $\alpha=1$ one obtains a better bound.
\begin{theorem}\label{thm:alpha1}
Let $\tilde{n}\in\mathbb{R}^+$ such that $\tilde{n}>6000$. Then 
$$\lceil\rho_{cont}(\tilde{n},\tilde{n})\rceil\le\rdb (\Gamma_{\lceil\tilde{n}\rceil,\lceil\tilde{n}\rceil})\le\lceil\rho_{cont}(\tilde{n},\tilde{n})\rceil+1.$$
\end{theorem}
By bootstrapping Theorem \ref{thm:alpha1} we conclude the possible range of $\alpha$ values.

\begin{corollary}\label{cor:bootstrap}
For any $1/2\le\alpha<1$ and $n,m\in\mathbb{Z}^+$ such that $m/n=\alpha$ and $n>8000$,
$$\lceil\rho_{cont}(n,m)\rceil\le\rdb (\Gamma_{n,m})\le\lceil\rho_{cont}(n,m)\rceil+2.$$
\end{corollary}

Note that one can computationally verify the correctness of the bounds of all values of $n$ smaller than $8000$. This algorithmic part is not immediate and will be a subject of a subsequent paper. 

It is easy to show an example where the ceiling plus two stated in Theorem \ref{theorem:Thm1} and  Corollary \ref{cor:bootstrap} is sharp. Take $(n,m)=(7,4)$. By checking all possible configurations, the optimal discrete double bubble perimeter is $19$ (see Figure \ref{fig:compsolver23}), whereas the continuous case solution has double bubble perimeter $2\cdot\sqrt{6(7+4)}\in(16,17)$;
\begin{figure}[H]
\includegraphics[width=6 cm]{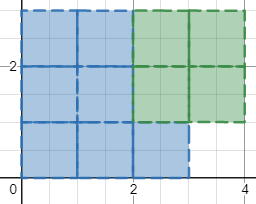}
\caption{\label{fig:compsolver23}}
\end{figure}

With the aide of a computer, one can find the optimal discrete double bubble perimeter for various integer volumes $(n,m)$. In Figure \ref{fig:compsolv}, purple represents solutions where $\rho_{DB}=\lceil\rho_{cont}\rceil$, teal represents $\rho_{DB}=\lceil\rho_{cont}\rceil+1$ and yellow represents $\rho_{DB}=\lceil\rho_{cont}\rceil+2$.
\begin{figure}[H]
\includegraphics[width=6 cm]{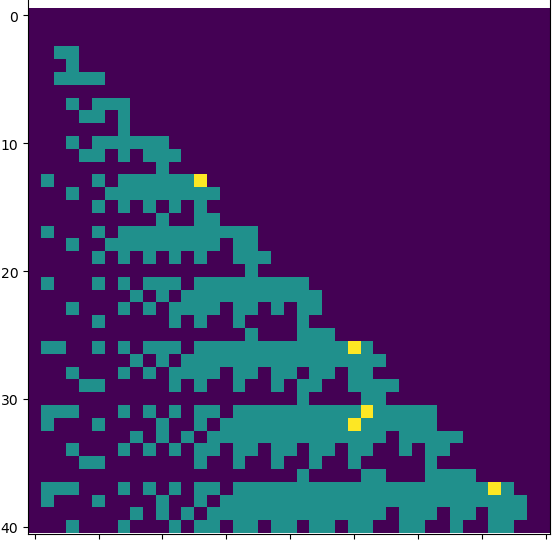}
\caption{\label{fig:compsolv}}
\end{figure}

One can observe that the requirement in Theorem \ref{thm:alpha1} for $n$ to be large enough is essential as one gets ceiling plus two at $(n,m)=(13,13)$. Here $\rho_{DB}$ is 27 where the continuous one is roughly 24.979.

Running the calculation for larger volumes one can see a path to improve upon our results. For some values of $\alpha$, ceiling plus two seems to not occur or at least be very rare. We leave this as an open problem.

\begin{figure}[H]
\includegraphics[width=8 cm]{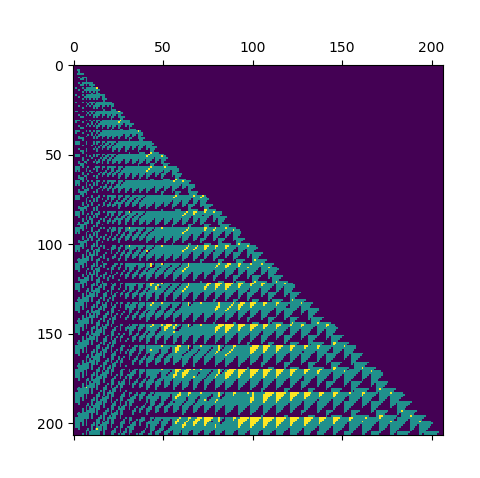}
\caption{\label{fig:compsolv2}}
\end{figure}


Paper Organization: Theorem \ref{theorem:Thm1} is proved in Section \ref{sec:proof1}. The methods used for the proof of Theorem \ref{theorem:Thm1} were too cumbersome to use to obtain the same bound for $\alpha\in(0.5,1]$.  The method we found to obtain this ceiling plus two bound for $\alpha\in(0.5,1]$ was strong enough to give us Theorem \ref{thm:alpha1}, from which Corollary \ref{cor:bootstrap} follows.  These are presented in Section \ref{sec3}.

\section{Proof of Theorem \ref{theorem:Thm1}}\label{sec:proof1}
In this section we say rectangle in $\mathbb{R}^2$ to represent sets of the form $[a,b]\times[c,d]$ with $a<b,c<d$ all in $\mathbb{R}$. A square in $\mathbb{R}^2$ assumes $|b-a|=|d-c|$. 

We prove Theorem \ref{theorem:Thm1} separately for two different cases of the continuous solution.
 
\subsection{Lattice Imposing Method, $\alpha\le\alpha_0$}

In this section we begin the proof of Theorem \ref{theorem:Thm1}; we aim to show through a sequence of lemmas that for $\alpha\leq\alpha_0$ the double bubble minimizer for the discrete case can be no more than two plus the ceiling of the continuous case double bubble minimizer.  To do this, we begin with two integers $m,n$ such that $\frac{m}{n}\leq\alpha_0$, and examine the shapes these two values would produce in the continuous case, and then adjust the lengths of the sides of the rectangles so that all of them are integer length.  We must, while doing this, be careful to not increase the perimeter by too much, of course; we must also ensure that we do not lose too much volume from our two sets.  It will be seen that this procedure may result in two sets whose volumes are greater than necessary, prompting use to develop, in case of this contingency, the most natural method to reduce their volumes to the desired values without increasing double bubble perimeter.  \\

\begin{lemma}

Let $(A,B)\in\Gamma_{\alpha}$ be such that $\mu(A)=n$, $\mu(B)=m$, and $A,B\in\mathfs{L}$, then $(A,B)\in\Gamma_{n,m}$ as well.  \\

\begin{proof}
This follows immediately from the definition of $\Gamma_{n,m}$.  \\
\end{proof}

\end{lemma}

In light of the previous lemma, we will assume for the remainder of section $2$ that at least one of the sides of either $A$, $B$, or any square or rectangle mentioned is not an integer, unless otherwise stated (here $(A,B)\in\Gamma_{\alpha}$).  \\

\begin{lemma}
\label{lemma:Lem2.2}
Let $S$ be a square in $\mathbb{R}^2$ with volume $V\in\mathbb{N}$, $V>1$, and therefore sides of length $\sqrt{V}$.  Let $\delta=\sqrt{V}-\left\lfloor\sqrt{V}\right\rfloor$ be the decimal part of $\sqrt{V}$.  If $\delta\in[0,0.5)$, then the figure, call it $S'$, obtained by increasing the horizontal sides of $S$ to have length $\left\lceil\sqrt{V}\right\rceil$ and decreasing the vertical sides of $S$ to have length $\left\lfloor\sqrt{V}\right\rfloor$ has perimeter at most $\rho(S)+2$, and is such that $\mu(S')\geq\mu(S)$.

\end{lemma}

\begin{proof}

The lemma is trivial if $\delta=0$ because in this case the sides of the square are already integers and taking the ceiling and floor of an integer does not change the value.  So let us assume that $\delta\in(0,0.5)$.  

Decreasing the vertical sides clearly does not increase perimeter.  Since $\delta\in(0,0.5)$, increasing the other two side lengths increases the perimeter by at most $2\cdot(1-\delta)<2$.  

At first glance, it may appear that this procedure will have decreased the volume because we did decrease the length of the vertical sides of the square.  However, notice that the original volume of $S$ was $V=\left\lfloor\sqrt{V}\right\rfloor^2+2\delta\cdot\left\lfloor\sqrt{V}\right\rfloor+\delta^2\in\mathbb{N}$, and $0<\delta^2<1$.  On the other hand, the new set has volume $V'=\left\lfloor\sqrt{V}\right\rfloor\cdot\left\lceil\sqrt{V}\right\rceil=\left\lfloor V\right\rfloor\cdot\left\lfloor V\right\rfloor+(1-\delta)\cdot\left\lfloor\sqrt{V}\right\rfloor+\delta\cdot\left\lfloor\sqrt{V}\right\rfloor\in\mathbb{N}$.  Also notice that since $\delta\in(0,0.5)$, we have that $(1-\delta)\cdot\left\lfloor\sqrt{V}\right\rfloor>\delta\cdot\left\lfloor\sqrt{V}\right\rfloor$.  This gives the following inequality:  
$$V-\delta^2=\left\lfloor\sqrt{V}\right\rfloor^2+2\delta\cdot\left\lfloor\sqrt{V}\right\rfloor<\left\lfloor\sqrt{V}\right\rfloor^2+(1-\delta)\left\lfloor\sqrt{V}\right\rfloor+\delta\left\lfloor\sqrt{V}\right\rfloor=V'$$

Finally, since $V'$ is an integer, we have $V=\left\lceil V-\delta^2\right\rceil\leq V'$.  Thus, our new volume is at least the original volume $V$.

\end{proof}

\begin{lemma}
\label{lemma:Lem2.3}

Let $S$ be a square in $\mathbb{R}^2$ with volume $V\in\mathbb{N}$, $V>1$, and therefore sides of length $\sqrt{V}$.  Let $\delta=\sqrt{V}-\left\lfloor\sqrt{V}\right\rfloor$ be the decimal part of $\sqrt{V}$.  If $\delta\in[0.5,1)$, then the figure, call it $S'$, obtained by increasing all sides of $S$ to have length $\left\lceil\sqrt{V}\right\rceil$ has perimeter at most $\rho(S)+2$, and is such that $\mu(S')\geq\mu(S)$.  \\

\end{lemma}

\begin{proof}

That this procedure increases the volume of $S$ is obvious since we are increasing its side lengths.  If we increase the length of each side of $S$ as indicated, we have increased four sides by $1-\delta\leq0.5$, which results in an increase in perimeter of at most $2$.

\end{proof}

The previous two lemmas give us a way to force our sets into $\mathfs{L}$ without increasing the perimeter too much and while also not decreasing the volume.  We have not, however, eliminated the possibility that in the construction provided by the last two lemmas we actually increase volume at some stage.  Our next lemma gives us a way of reducing the volume if we have more than we need after the process given in Lemmas \ref{lemma:Lem2.2} and \ref{lemma:Lem2.3}.  \\

\begin{lemma}
\label{lemma:Lem2.4}
Let $A\in\mathfs{L}$ be a rectangle.  If $l,k\in\mathbb{N}$, $\mu(A)=k>1$, and $l<k$, then we can reduce the volume of $A$ until we have a set $A'\subset A$ such that $\mu(A')=l$, $\rho(\partial A')\leq\rho(\partial A)$, and $A'\in\mathfs{L}$.  

\end{lemma}

\begin{proof}
Consider the following figure:  \\

\begin{figure}[H]

\begin{center}
\begin{tikzpicture}[scale=0.7]
\draw (-2,0) -- (6,0);
\draw (0,-2) -- (0,6);

\draw (-0.2,1) -- (0.2,1);
\draw (-0.2,0.5) -- (0.2,0.5);
\draw (-0.2,1.5) -- (0.2,1.5);
\draw (-0.2,2) -- (0.2,2);
\draw (-0.2,2.5) -- (0.2,2.5);
\draw (-0.2,3) -- (0.2,3);
\draw (-0.2,3.5) -- (0.2,3.5);
\draw (-0.2,4) -- (0.2,4);
\draw (-0.2,4.5) -- (0.2,4.5);

\draw (0.5,-0.2) -- (0.5,0.2);
\draw (1,-0.2) -- (1,0.2);
\draw (1.5,-0.2) -- (1.5,0.2);
\draw (2,-0.2) -- (2,0.2);
\draw (2.5,-0.2) -- (2.5,0.2);
\draw (3,-0.2) -- (3,0.2);
\draw (3.5,-0.2) -- (3.5,0.2);
\draw (4,-0.2) -- (4,0.2);
\draw (4.5,-0.2) -- (4.5,0.2);

\draw [blue, thick] (0,0) -- (0,4);
\draw [blue, thick] (0,4) -- (4,4);
\draw [blue, thick] (4,4) -- (4,0);
\draw [blue, thick] (4,0) -- (0,0);

\draw (2,2.7) node{$A$};
\draw (2,1.8) node{$\mu(A)=k$};

\end{tikzpicture}
\caption{\label{fig:Figure5}}

\end{center}

\end{figure}

We merely have to remove volume one square at a time from the left side of this shape beginning at the bottom until we have volume $l$ remaining.  Only if we remove an entire column of unit squares does the perimeter change, and in that case it decreases instead of increases.  If we do remove an entire column, we continue removing the next column on the left side of this new figure.  In this way, we will eventually reach volume $l$ without increasing perimeter. \\

\end{proof}

Before looking at Lemma \ref{lemma:Lem2.5}, recall that if $m,n\in\mathbb{N}$, and $\frac{m}{n}\leq\alpha_0$, then the continuous double bubble configuration is as follows:  \\

\begin{center}
\begin{figure}[H]
\begin{tikzpicture}[scale=0.7]
\draw[blue,  thin] (0,4) to (4,4);
\draw[blue,  thin] (0,0) to (4,0);
\draw[blue,  thin] (0,0) to (0,4);
\draw[blue,  thin] (4,0) to (4,4);
\draw[blue, thin] (3,3) to (3,4);
\draw[blue, thin] (3,3) to (4,3);
\draw (2,0.35) node{$\sqrt{m+n}$};
\draw (1.05,2) node{$\sqrt{m+n}$};
\draw (2.45,3.5) node{$\sqrt{m}$};
\draw (3.4,2.65) node{$\sqrt{m}$};

\draw (0.3,3.6) node{$A$};
\draw (3.6,3.6) node{$B$};

\end{tikzpicture}
\caption{\label{fig:Figure6}}

\end{figure}
\end{center}

In the following lemma, we will consider $A\cup B$ in the figure above to be a square in its own right, even though it technically doesn't contain the subset of $\mathbb{R}^2$ that is $\partial A\cap\partial B$.

\begin{lemma}
\label{lemma:Lem2.5}
Let $(A,B)\in\Gamma_{\alpha}$ be such that $\mu(A)=n$, $\mu(B)=m$, where $m,n\in\mathbb{N}^+$ are integers such that $\frac{m}{n}\leq\alpha_0$, and $a=\sqrt{n+m}-\left\lfloor\sqrt{n+m}\right\rfloor$, and $b=\sqrt{m}-\left\lfloor\sqrt{m}\right\rfloor$. Then we can find two new sets $(A',B')$ such that $\rho_{DB}(A',B')\leq\left\lceil\rho_{cont}(A,B)\right\rceil+2$, and such that $\mu(A'\cup B')=\mu(A\cup B)$ and $\mu(B')=\mu(B)$.  
\end{lemma}

\begin{proof}
The result follows from Lemmas \ref{lemma:Lem2.2}, \ref{lemma:Lem2.3}, and \ref{lemma:Lem2.4}.  We have two options for the value of $a$, (either $a\in[0,0.5)$ or $a\in(0.5,1)$), and likewise have two options for the value of $b$.  This gives a total of four combinations of $a$ and $b$.  

Take, for example, $a\in[0,0.5)$ and $b\in[0,0.5)$; Lemma \ref{lemma:Lem2.2} produces for us two new shapes $A'$ and $B'$ such that the side lengths of $B'$ and $A'\cup B'$ will be, respectively, $\left\lceil\sqrt{m}\right\rceil\cdot\left\lfloor\sqrt{m}\right\rfloor$ and $\left\lceil\sqrt{m+n}\right\rceil\cdot\left\lfloor\sqrt{m+n}\right\rceil$.  Furthermore, we know from the same lemma that $\mu(B')\geq m$ and $\mu(A'\cup B')\geq m+n$.  Referring to the figure above, notice here that when adjusting the shape of $B$, we only need to move the part of $\partial B$ that appears in the closure of $A\cup B$.  This means that the double bubble perimeter can increase by at most $1$.  By adjusting the sides of $A\cup B$, Lemma \ref{lemma:Lem2.2} tells us that the double bubble perimeter can increase by at most $2$.  Therefore, overall we have increased the double bubble perimeter by at most $3$.  Since we started out with a double bubble perimeter $\rho_{DB}(A,B)\notin\mathbb{N}$, and we now have a double bubble perimeter $\rho_{DB}(A',B')\in\mathbb{N}$, and $\rho_{DB}(A',B')\leq\rho_{DB}(A,B)+3$, it follows that $\rho_{DB}(A',B')\leq\left\lfloor\rho_{cont}(A,B)+3\right\rfloor=\left\lceil\rho_{cont}(A,B)\right\rceil+2$.  

There are three more possible combinations for $a$ and $b$.  We next take when $a\in[0,0.5)$ and $b\in(0.5,1)$.  Lemma \ref{lemma:Lem2.2} provides us with a new shape for $A\cup B$ with side lengths $\left\lceil\sqrt{m+n}\right\rceil\cdot\left\lfloor\sqrt{m+n}\right\rfloor$, and Lemma \ref{lemma:Lem2.3} provides us with a new shape for $B$ with side lengths $\left\lceil\sqrt{m}\right\rceil\cdot\left\lceil\sqrt{m}\right\rceil$.  These two lemmas combined tell us that the increase in double bubble perimeter is not more than $2$ in this case, and also that the volumes of both $B$ and $A\cup B$ have not decreased.  

For $a\in(0.5,1)$ and $b\in(0,0.5)$ Lemmas \ref{lemma:Lem2.2} and Lemma \ref{lemma:Lem2.3}  provide two shapes that have side lengths $\left\lceil\sqrt{m+n}\right\rceil\cdot\left\lceil\sqrt{m+n}\right\rceil$ and $\left\lceil\sqrt{m}\right\rceil\cdot\left\lfloor\sqrt{m}\right\rfloor$.  Finally when both $a,b\in(0.5,1)$ we are supplied with two shapes from Lemma \ref{lemma:Lem2.3} that have side lengths $\left\lceil\sqrt{m+n}\right\rceil\cdot\left\lceil\sqrt{m+n}\right\rceil$ and $\left\lceil\sqrt{m}\right\rceil\cdot\left\lceil\sqrt{m}\right\rceil$.  Again, we know that these shapes have not increased the double bubble perimeter by more than $3$, and we know that the volume of the smaller shape is at least $m$ and the joint volume of both shapes is at least $m+n$.  Again, the same argument as above ensures that since we know that we have increased the double bubble perimeter by at most $3$, this means that $\rho_{DB}(A',B')\leq\left\lceil\rho_{cont}(A,B)\right\rceil+2$. 

Applying Lemma \ref{lemma:Lem2.4} to $B'$ first, and at each step that we remove volume from $B'$ adding it to $A'$, we can find two disjoint sets, which by abuse of notation we continue to call $A'$ and $B'$, such that $\mu(B')=m$, and $\mu(A'\cup B')\geq n+m$.  We can then apply Lemma \ref{lemma:Lem2.4} to $A'\cup B'$, being careful to remove volume from $A'$ but not $B'$ until $\mu(A'\cup B')=n+m$.  Since we already know that $\mu(B')=m$, this means that $\mu(A')=n$, as was desired.  The result will look like Figure \ref{fig:Figure6}.

\begin{figure}[H]

\begin{center}
\begin{tikzpicture}[scale=0.6]
\draw (-2,0) -- (6,0);
\draw (0,-2) -- (0,6);

\draw (-0.2,1) -- (0.2,1);
\draw (-0.2,0.5) -- (0.2,0.5);
\draw (-0.2,1.5) -- (0.2,1.5);
\draw (-0.2,2) -- (0.2,2);
\draw (-0.2,2.5) -- (0.2,2.5);
\draw (-0.2,3) -- (0.2,3);
\draw (-0.2,3.5) -- (0.2,3.5);
\draw (-0.2,4) -- (0.2,4);
\draw (-0.2,4.5) -- (0.2,4.5);

\draw (0.5,-0.2) -- (0.5,0.2);
\draw (1,-0.2) -- (1,0.2);
\draw (1.5,-0.2) -- (1.5,0.2);
\draw (2,-0.2) -- (2,0.2);
\draw (2.5,-0.2) -- (2.5,0.2);
\draw (3,-0.2) -- (3,0.2);
\draw (3.5,-0.2) -- (3.5,0.2);
\draw (4,-0.2) -- (4,0.2);
\draw (4.5,-0.2) -- (4.5,0.2);

\draw [blue, thick] (0,1) -- (0,4);
\draw [blue, thick] (0,4) -- (4,4);
\draw [blue, thick] (4,4) -- (4,0);
\draw [blue, thick] (4,0) -- (0.5,0);
\draw [blue, thick] (0.5,0) -- (0.5,1);
\draw [blue, thick] (0.5,1) -- (0,1);

\draw [red, thick] (2.5,4) -- (2.5,3);
\draw [red, thick] (2.5,3) -- (3,3);
\draw [red, thick] (3,3) -- (3,2.5);
\draw [red, thick] (3,2.5) -- (4,2.5);

\draw (2,2) node{$A'$};
\draw (3.4,3.4) node{$B'$};

\end{tikzpicture}

\end{center}
\caption{}
\label{fig:Figure6}
\end{figure}

This concludes all of the possibilities for $\frac{m}{n}\leq\alpha_0$.

\end{proof}

 \subsection{Lattice Imposing Method, $\alpha_0<\alpha\leq0.5$}
Now, we prove that the discrete double bubble minimizer is no more than $\left\lceil\rho_{cont}\right\rceil+2$ when $\frac{m}{n}=\alpha\in(\alpha_0,0.5]$.  For these values of $\alpha$ the continuous minimizer looks like the following: 

\begin{figure}[H]

\begin{center}
\begin{tikzpicture}

\draw[blue,  thin] (5,4) to (9,4);
\draw[blue,  thin] (5,0) to (9,0);
\draw[blue,  thin] (5,0) to (5,4);
\draw[blue,  thin] (9,0) to (9,4);
\draw (7,0.25) node{$\sqrt{n}$};
\draw (5.35,2) node{$\sqrt{n}$};
\draw (7,-0.5) node{$\alpha\in [\frac{688-480\sqrt{2}}{49},.5]$};
\draw[blue,  thin] (9,1) to (10.5,1);
\draw[blue,  thin] (9,3) to (10.5,3);
\draw[blue,  thin] (10.5,1) to (10.5,3);
\draw (10.05,2) node{$\sqrt{2m}$};
\draw (9.80,3.35) node{$\frac{\sqrt{2m}}{2}$};

\end{tikzpicture}
\caption{\label{fig:Figure8}}

\end{center}

\end{figure}

\begin{lemma}
\label{lemma:Lem2.7}
Let $(A,B)\in\Gamma_{\alpha}$ be such that $\mu(A)= n$, $\mu(B)= m$, $\alpha_0<\frac{m}{n}\leq0.5$, $a=\sqrt{n}-\left\lfloor\sqrt{n}\right\rfloor$ and $b=\sqrt{2m}-\left\lfloor\sqrt{2m}\right\rfloor$.  Then we can find $(A',B')$ so that $\mu(A')=n$, $\mu(B')=m$, and $\rho_{DB}(A',B')\leq\left\lceil\rho_{cont}(A,B)\right\rceil+2$.  

\end{lemma}

\begin{proof}

First, when $a\in(0,0.5)$, Lemma \ref{lemma:Lem2.2} gives us a set (a rectangle, in fact) whose side lengths are $\left\lceil\sqrt{n}\right\rceil$ and $\left\lfloor\sqrt{n}\right\rfloor$.  Note here that it is important to increase the side of $A$ that has shared boundary with $B$.  This way if we must increase this side of $B$, we have ensured that no side of $B$ becomes longer than the side of $A$ with which it has shared boundary.  

Now, if $a\in(0.5,1)$, we use Lemma \ref{lemma:Lem2.3} to find a figure whose side lengths are $\left\lceil\sqrt{n}\right\rceil\cdot\left\lceil\sqrt{n}\right\rceil$.  In either of these cases, we have increased the perimeter by at most $2$.  Therefore, if we can show that by adjusting $B$ we increase the double bubble perimeter by at most $1$, we are done.  

The proof depends on whether $\left\lfloor\sqrt{2Y}\right\rfloor$ is even or odd as well as on $b$.  Suppose first that $\left\lfloor\sqrt{2Y}\right\rfloor$ is odd.  Then if $b\in(0,0.5)$, the decimal part of $\frac{\sqrt{2Y}}{2}$ is between $0.5$ and $0.75$.  This means that if we increase the two sides of length $\frac{\sqrt{2Y}}{2}$ to $\left\lceil\frac{\sqrt{2Y}}{2}\right\rceil$, and decrease the sides of length $\sqrt{2Y}$ to $\left\lfloor\sqrt{2Y}\right\rfloor$, we have increased the double bubble perimeter by at most one.  We need to ensure that we have not decreased the volume.  

We have increased the volume by at least $0.25\cdot\left\lfloor\sqrt{2Y}\right\rfloor$ but decreased it by at most $0.5\cdot\frac{\sqrt{2Y}}{2}$.  At first glance things look bleak as it may appear that the lost volume outweighs the volume gained.  However, recall that our original volume was an integer, and the volume after adjusting the side lengths is also an integer.  Therefore, for the volume to have reduced, the difference $0.25\cdot\left\lfloor\sqrt{2Y}\right\rfloor-0.5\cdot\frac{\sqrt{2Y}}{2}=0.25\left(\left\lfloor\sqrt{2Y}\right\rfloor-\sqrt{2Y}\right)$ must be at least one, which, clearly, it is not.  

Let us now take the case when $\left\lfloor\sqrt{2Y}\right\rfloor$ is still odd, but $b\in(0.5,1)$.  Then, the decimal part of $\frac{\sqrt{2Y}}{2}$ is between $0.75$ and $1$.  Then we can simply increase all side lengths to the next integer up, i.e. we increase $\sqrt{2Y}$ to $\left\lceil\sqrt{2Y}\right\rceil$ and increase $\frac{\sqrt{2Y}}{2}$ to $\left\lceil\frac{\sqrt{2Y}}{2}\right\rceil$.  With reference to the figure above, it is easy to see that we have increased the perimeter by at most $1$, which is what we wanted.  

We now examine the cases when $\left\lfloor\sqrt{2Y}\right\rfloor$ is even.  Take first the case when $b\in(0,0.5)$.  Then the decimal part of $\frac{\sqrt{2Y}}{2}$ is between $0$ and $0.25$.  Therefore, we increase the side of length $\sqrt{2Y}$ to $\left\lceil\sqrt{2Y}\right\rceil$ and decrease the sides of length $\frac{\sqrt{2Y}}{2}$ to $\left\lfloor\frac{\sqrt{2Y}}{2}\right\rfloor$.  This means that we have increased the perimeter by at most $1$.  As for the volume, we have increased it by at least $0.5\cdot\left\lfloor\frac{\sqrt{2Y}}{2}\right\rfloor$, whereas we decrease the volume by at most $0.25\cdot\sqrt{2Y}$.  Again, noting that the volumes of both the original figure as well as that of the post-procedure figure are both integers guarantees that we have not lost anything.  

Finally, we are in the last case, which is when $\left\lfloor\sqrt{2Y}\right\rfloor$ is even, and $b\in(0.5,1)$.  In this case, we know that the decimal part of $\frac{\sqrt{2Y}}{2}$ is between $0.25$ and $0.5$.  Here, we reduce the side length $\sqrt{2Y}$ to $\left\lfloor\sqrt{2Y}\right\rfloor$ and increase the two sides of length $\frac{\sqrt{2Y}}{2}$ to $\left\lceil\frac{\sqrt{2Y}}{2}\right\rceil$.  Clearly we have increased the double bubble perimeter by at most $1$.  We have increased the volume by at least $0.5\cdot\left\lfloor\frac{\sqrt{2Y}}{2}\right\rfloor$ and decreased it by at most $1\cdot{\frac{\sqrt{2Y}}{2}}$.  We again rely on the fact that both the original volume and the volume of the figure after our procedure are integers to guarantee that the volume has not decreased.  

This completes all possibilities for the case when $\alpha\in(\alpha_0,0.5]$.
\end{proof}
Having verified all of $\alpha\in(0,0.5]$ we conclude the proof of Theorem \ref{theorem:Thm1} \qed

\section{Proof of Theorem \ref{thm:alpha1} and Corollary \ref{cor:bootstrap}}\label{sec3}

For the final values $\alpha\ge0.5$, there are too many parameters that create too many cases to check using the previous methods, thus we present here another technique to address the remaining values of $\alpha$. Instead, we will show that for $\alpha=1$, and $n$ large enough, we can always find two rectangles whose side lengths are integers, and whose double bubble perimeter is at most the ceiling of the continuous case plus one.  From this, we will argue that for any $\alpha\in[0.5,1)$, there is a discrete configuration whose double bubble perimeter is at most the ceiling of the continuous case plus two.  

\begin{proof}[Proof of Theorem \ref{thm:alpha1}]
Recall that given two volumes, $m,n$ such that $\frac{m}{n}\in[0.5,1]$, the shape that minimizes the continuous double bubble perimeter is of the following form:

\begin{figure}[H]
\begin{tikzpicture}

\draw[blue,  thick] (-4,0) to (4,0);
\draw[blue,  thick] (-4,0) to (-4,4);
\draw[blue,  thick] (-4,4) to (4,4);
\draw[blue,  thick] (4,4) to (4,0);

\draw[blue,  thick] (0,4) to (0,0);

\draw (-2,2) node{$\mu(A)=n$};
\draw (2,2) node{$\mu(B)=m$};

\draw (-2,-0.3) node{$x$};
\draw (2,-0.3) node{$y$};
\draw (-4.25,2) node{$z$};

\end{tikzpicture}
\caption{\label{fig:Figure9}}
\end{figure}

The values of $x,y$, and $z$ are given in Figure \ref{fig:Figure4} as $(x,y,z)=\left(\sqrt{\frac{3n^{2}}{2(n+m)}},\sqrt{\frac{3m^{2}}{2(n+m)}},\sqrt{\frac{2(n+m)}{3}}\right)$, and the double bubble perimeter is $\rdb=2\sqrt{6(n+m)}$.  

If we forget for a moment about minimizing double bubble perimeter and focus on Figure \ref{fig:Figure9}, we can see that its double bubble perimeter is $3z+2(x+y)$.  Let us fix $n=m$, and consider only when $\alpha=1$.  The expression representing the double bubble perimeter becomes $2\sqrt{12n}$.  Suppose we find $x,y,z\in\mathbb{Z}^+$ such that $xz\geq n$, $yz\geq n$, and $3z+2(x+y)\leq\left\lceil2\sqrt{6(n+m)}\right\rceil+1$.  Then we will have completed the proof of Theorem \ref{thm:alpha1}, for, although $xz,yz$ may be too large originally, by Lemma \ref{lemma:Lem2.4} we can reduce the volumes in such a way as to not increase double bubble perimeter.  

Notice that the three dimensional set formed by the inequalities $xz\geq n$ and $yz\geq n$ is convex.  Let $\mathfs{C}=\{(x,y,z)\in\mathbb{R}^3|xz\geq n,yz\geq n\}$.  Any point $(x_0,y_0,z_0)$ in $\mathfs{C}$ represents a configuration of the form presented in Figure \ref{fig:Figure9}, although it may be the case that the volumes are too large.  What we wish to show is that the number $\min\{3z+2(x+y)|(x,y,z)\in\mathfs{C}, x,y,z\in\mathbb{Z}^+\}$ is at most $\left\lceil2\sqrt{6(n+m)}\right\rceil+1$, i.e. at most one more than the ceiling of the continuous case double bubble perimeter.  Note that this minimum clearly exists as there are certainly positive integers $x,y,z$ that satisfy the inequalities defining $\mathfs{C}$.  

Consider the planes $\mathfs{P}_q$ of the form $3z+2(x+y)=\Big\lceil2\sqrt{12n}\Big\rceil+q$, where $q\in\mathbb{Z}^+$, and let $\tilde{\mathfs{P}}_0$ be the plane given by $3z+2(x+y)=0$.  For given $q$, a point that is in the intersection of $\mathfs{P}_q$ and $\mathfs{C}$ represents a valid configuration like that in Figure \ref{fig:Figure9} and has double bubble perimeter $\Big\lceil2\sqrt{12n}\Big\rceil+q$.  If there is a point in $\mathfs{P}_q\cap\mathbb{Z}^3\cap\mathfs{C}$, then that is a valid configuration that also happens to have integer lengths for all of its sides.  We need to find such a point when $q=1$.  

Looking at the equation for the plane $\mathfs{P}_q$, we notice that the vector $(2,2,3)$ is orthogonal to the plane.  It is easily seen, by taking the dot product, that both the vectors $v_1=(0,3,-2)$ and $v_2=(1,-1,0)$ are orthogonal to $(2,2,3)$. Since $v_1$ and $v_2$ are linearly independent they form a basis for $\tilde{\mathfs{P}}_0$. Moreover since $v_1$ and $v_2$ have the smallest Euclidean norm in $\tilde{\mathfs{P}}_0\cap\mathbb{Z}^3$, they span this lattice.  

Now we are going to project the affine lattice $\mathfs{P}_1\cap\mathbb{Z}^3$ onto the $xy$-plane.  This will result in a lattice, call it $\mathfs{Z}$, provided $\mathfs{P}_1\cap\mathbb{Z}^3\neq\emptyset$ (which it isn't, as we shall see).  Then we will project $\mathfs{P}_1\cap\{(x,y,z)|xz=n\}$, and $\mathfs{P}_1\cap\{(x,y,z)|yz=n\}$ onto the $xy$-plane.  These will create two curves that intersect each other in such a way as to create one convex, bounded set in $\mathbb{R}^2$; let's denote this convex and bounded set by $\mathfs{R}$.  We will then show that for $n$ large enough, we can fit a parallelogram inside $\mathfs{R}$ that is big enough to guarantee it contains a lattice point from $\mathfs{Z}$.  This point guarantees that we have a point in $\mathfs{P}_1\cap\mathbb{Z}^3\cap\mathfs{C}$.  Finally, this means that we have a configuration representing two sets whose volumes are at least $n$, and whose double bubble perimeter is at most $\Big\lceil2\sqrt{12n}\Big\rceil+1$.  

Let us first convince ourselves that $\mathfs{P}_1\cap\mathbb{Z}^3\neq\emptyset$.  The equation representing $\mathfs{P}_1$ is $3z+2(x+y)=\left\lceil2\sqrt{12n}\right\rceil+1$, which gives $z=\frac{\left\lceil2\sqrt{12n}\right\rceil-2(x+y)+1}{3}$.  Since we are dividing by $3$, and the numerator is an integer when $x$ and $y$ are, and $\mathbb{Z}/3\mathbb{Z}$ has three equivalence classes, the integers $\left\lceil2\sqrt{12n}\right\rceil-2(x+y)+1$ when $(x,y)=(0,0),(0,1),(0,2)$ are all in different equivalence classes, and therefore one of them must be an integer that is divisible by $3$, which makes $z$ an integer as well as $x$ and $y$.  This means that $\mathfs{P}_1\cap\mathbb{Z}^3\neq\emptyset$, as we wished.  Since we know that $\mathfs{P}_1\cap\mathbb{Z}^3\neq\emptyset$, we can project its elements onto the $xy$-plane.  

The above argument tells us more; it tells us that $\mathfs{P}_1$ is $\tilde{\mathfs{P}}_0$ plus some vector —either $\left(0,0,\frac{\left\lceil2\sqrt{12n}\right\rceil+1}{3}\right)$, $\left(0,1,\frac{\left\lceil2\sqrt{12n}\right\rceil-1}{3}\right)$, or $\left(0,2,\frac{\left\lceil2\sqrt{12n}\right\rceil-3}{3}\right)$—which we call a shift vector.  Thus, despite the variability of $n$, we know that the projection of the lattice $\mathbb{Z}^3\cap\mathfs{P}_1$ onto the $xy$-plane can be located by simply adding vectors in the projection of $\tilde{\mathfs{P}}_0$ to the projection of one of these three shift vectors, i.e. regardless of the value of $n$, the lattice we are interested in is one of three options.  

Now, let us examine the projections of $\mathfs{P}_1\cap\{(x,y,z)|xz=n\}$ and $\mathfs{P}_1\cap\{(x,y,z)|yz=n\}$ onto the $xy$-plane.  We already saw that the equation for $\mathfs{P}_1$ can be given by $z=\frac{\left\lceil2\sqrt{12n}\right\rceil-2(x+y)+1}{3}$.  Similarly, the other surfaces are given by the equations $z=\frac{n}{x}$, and $z=\frac{n}{y}$.  Thus the intersection is given by the equalities

\begin{equation}\label{eq:curves}
\begin{aligned}
&\frac{n}{x}=\frac{\left\lceil2\sqrt{12n}\right\rceil-2(x+y)+1}{3},\\
&\frac{n}{y}=\frac{\left\lceil2\sqrt{12n}\right\rceil-2(x+y)+1}{3}.
\end{aligned}
\end{equation}

We can then project both of these equations as well as $\mathfs{P}_1\cap\mathbb{Z}^3$ onto the $xy$-plane to get the right image in Figure \ref{fig:helper11111}. 
The projection of $\mathfs{P}_1\cap\mathbb{Z}^3$ onto the $xy$-plane depends on $n$, and is spanned by a new basis $\{(1,-1),(0,3)\}$ plus a constant shift.   The constant shift is given by the projection of one of $\left(0,0,\frac{\left\lceil2\sqrt{12n}\right\rceil+1}{3}\right)$, $\left(0,1,\frac{\left\lceil2\sqrt{12n}\right\rceil-1}{3}\right)$, or $\left(0,2,\frac{\left\lceil2\sqrt{12n}\right\rceil-3}{3}\right)$ onto the $xy$-plane, i.e. the shift will be given by one of $(0,0)$, $(0,1)$, or $(0,2)$.  Note that the shift will not affect our analysis as we just show that we can embed a large enough shape between the projected curves, such that an integer lattice point must be contained inside it. Adding twice the first vector to the second gives the basis $\{(1,-1),(2,1)\}$.  It is easier to use a linear transformation so that the projected lattice becomes a shift of $\mathbb{Z}^2$ (see Figure \ref{fig:helper11111}), and the basis vectors become $\{(1,0),(0,1)\}$.  The matrix $\mathcal{L}$ that takes our basis onto the standard basis and its inverse are:

\begin{center}
$\mathcal{L}=\begin{pmatrix}1/3&&1/3\\1/3&&-2/3\end{pmatrix},\mathcal{L}^{-1}=\begin{pmatrix}2&&1\\1&&-1\end{pmatrix}$
\end{center}

The matrix $\mathcal{L}$ applied to the three possible shift vectors results in new shift vectors of $(0,0)$, $(1/3,-2/3)$, and $(2/3,-4/3)$.  Up to now, we have found a lattice $\mathfs{P}_1\cap\mathbb{Z}^3$ that is a vector shift away from the lattice $\tilde{\mathfs{P}}_0\cap\mathbb{Z}^3$ — there are only three possible such vector shifts — and projected this lattice onto the $xy$-plane.  To this projected lattice, we apply a linear transformation so that it will be a shifted copy of $\mathbb{Z}^2$.  

We apply the transformation $\mathcal{L}^{-1}$ to the curves given in Equation \ref{eq:curves}, which results in the new equations given by:

\begin{equation}\label{eq:TransformedCurves}
\begin{aligned}
&\frac{n}{2x-y}=\frac{\left\lceil2\sqrt{12n}\right\rceil-2(3x)+1}{3},\\
&\frac{n}{x+y}=\frac{\left\lceil2\sqrt{12n}\right\rceil-2(3x)+1}{3}.
\end{aligned}
\end{equation}

If we can find a point in our shifted lattice that lies in the bounded, convex region formed by these two curves, then we will have found a point that corresponds to a configuration of two sets whose volumes are both at least $n$, and whose double bubble perimeter is at most one more than the ceiling of the continuous case double bubble perimeter for two sets of volume $n$.  To find this point, we will show that we can fit a parallelogram inside this region of sufficient size as to guarantee that, no matter which of the three possible shifts is the correct one, there must be a shifted lattice point inside the parallelogram,.  We will then show that the region formed by these two curves increases monotonically with $n$, and therefore, if our parallelogram of sufficient size fits inside the region for $n'$, it must also fit inside the region for all $n\geq n'$.  

Before we proceed further, we would like to note that the corners of our region can be seen to lie on $y=.5x$ with this simple arguement. Namely the corners occur at coordinates where both equations in \ref{eq:TransformedCurves} hold. Therefore we have $\frac{n}{2x-y}=\frac{n}{x+y}$, which immediattely rearranges into $y=.5x$.

Solving Equation \ref{eq:TransformedCurves} for $y$ gives:

\begin{equation}\label{eq:FunctionOfCurves}
\begin{aligned}
&y_1=\frac{2x(\left\lceil2\sqrt{12n}\right\rceil-6x+1)-3n}{\left\lceil2\sqrt{12n}\right\rceil-6x+1},\\
&y_2=\frac{x(-\left\lceil2\sqrt{12n}\right\rceil+6x-1)+3n}{\left\lceil2\sqrt{12n}\right\rceil-6x+1}.
\end{aligned}
\end{equation}

The equation representing $y_1$ forms the top part of the region in question, and $y_2$ forms the bottom.  
In Figure \ref{fig:helper11111} one can observe the curves and possible integer lattice shifts, before and after the linear transformation $\mathcal{L}$. 

\begin{figure}[H]
\includegraphics[width=6 cm]{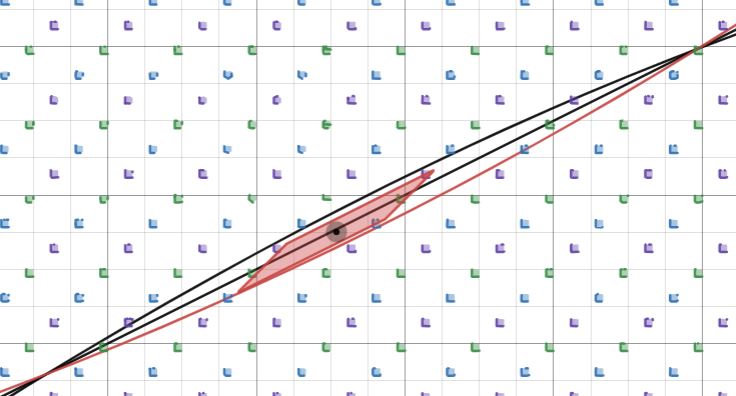}
\includegraphics[width=5 cm]{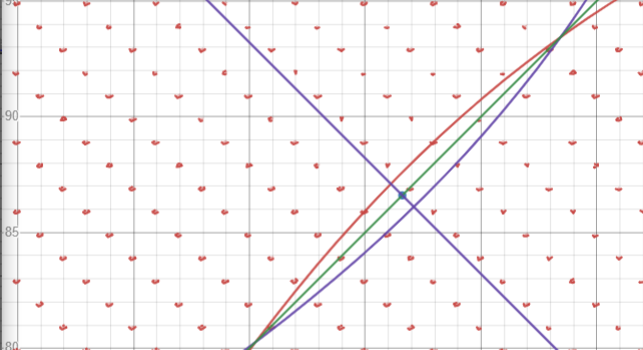}
\caption{The untransformed (right) and transformed (left) problems}
\label{fig:helper11111}
\end{figure}

Note that by equating the LHS of the two equations in \eqref{eq:TransformedCurves}, the intersection of the curves $y_1(x)$ and $y_2(x)$ is on the line $y=x/2$ for all $n$ (see Figure \ref{fig:lineintersection} for a graphical representation). 
\begin{figure}[H]
\includegraphics[width=6cm]{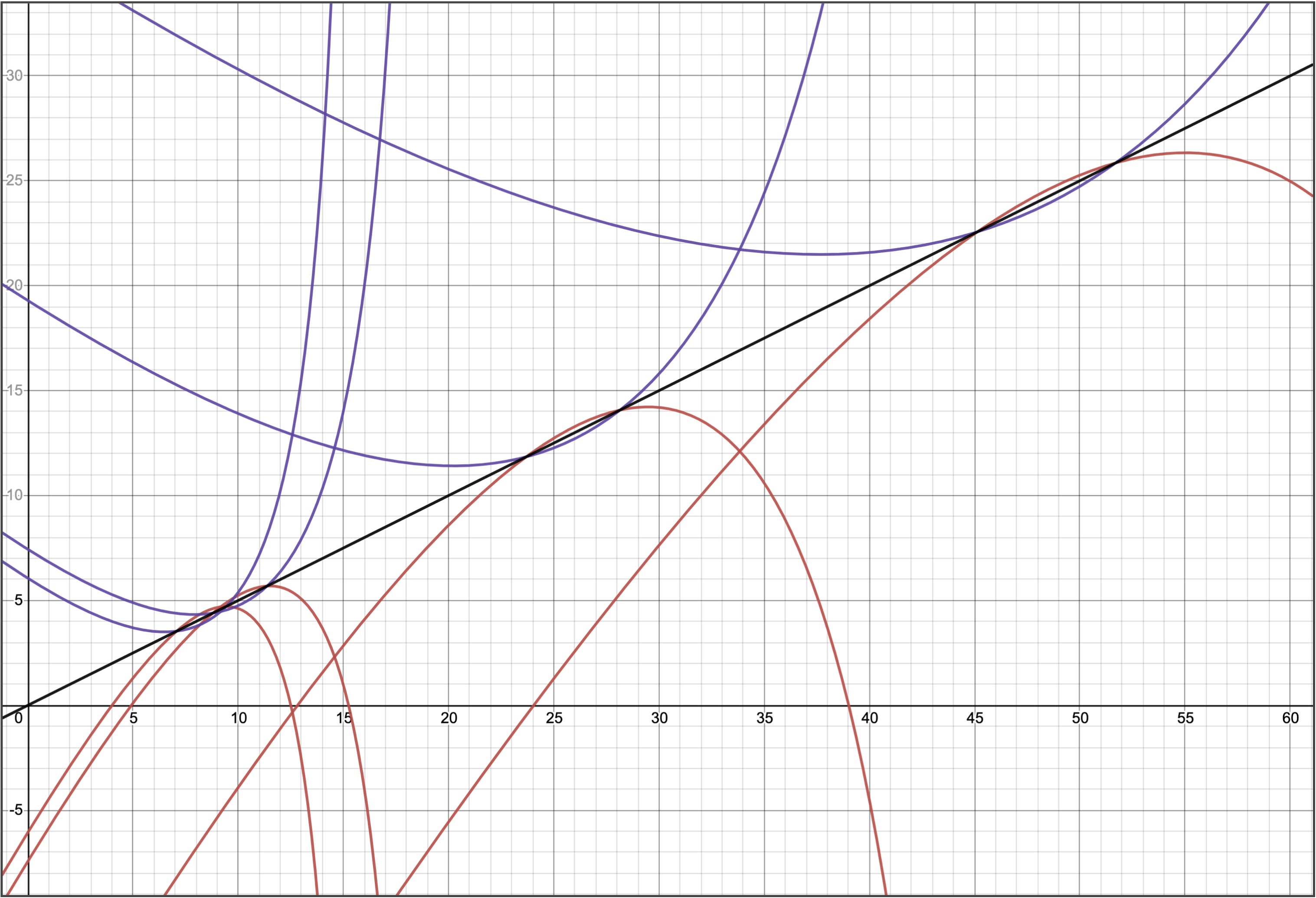}
\caption{Translation along the line $y=x/2$}
\label{fig:lineintersection}
\end{figure}

Thus if we translate the curves along this line, and fit a proper parallelogram centered along the line, then we have found an integer point along one of the possible shifted lattices.

We shift $x\to x+\frac{2+\sqrt{12n}}{6}$ and $y\to y+\frac{2+\sqrt{12n}}{12}$. Moreover for ease of calculations we will work with a smaller region which is obtained by dropping the ceiling function. We obtain the new curves:
\begin{equation}
\begin{aligned}
&\tilde{y}_{1}=\frac{\sqrt{3n}(1+2x)-1-10x-24x^2}{4\sqrt{3n}-12x-2}\\
&\tilde{y}_{2}=\frac{\sqrt{3n}(-1+2x)+1+8x+12x^2}{4\sqrt{3n}-12x-2}.
\end{aligned}
\end{equation}


\begin{figure}[H]
\includegraphics[width=6cm]{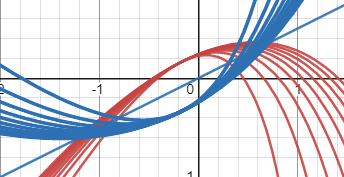}
\caption{Increasing region size with $n$}
\label{fig:helper11}
\end{figure}

Now we want to show that for $N\ge n$ the region between the curves $\tilde{y}_{1}$ and $\tilde{y}_{2}$ for $N$ contains the region for $n$. Note that the corners of the region corresponding to $n$ lie at $x=\frac{\pm\sqrt{1+8\sqrt{3n}}-3}{12}$, thus the corners indeed monotonically grow. Now we must check whether within the smaller region  $\tilde{y}_{1}(N,x)\ge \tilde{y}_{1}(n,x)$ and  $\tilde{y}_{2}(N,x)\le \tilde{y}_{2}(n,x)$. Looking for what values of $\{N,n,x\}$ the above conditions hold.
We find that, given $N\ge n\ge 1$,
$\tilde{y}_{1}(N,x)\ge \tilde{y}_{1}(n,x)$ for $x\in (-\infty,\sqrt{\frac{n}{3}}-\frac{1}{6})$. The same bounds hold for $\tilde{y}_{2}(N,x)\le \tilde{y}_{2}(n,x)$. 
The right corner of the small region for $n$ is $\frac{\sqrt{1+8\sqrt{3n}}-3}{12}$. 
Indeed for all $n\ge 1$, $\frac{\sqrt{1+8\sqrt{3n}}-3}{12}\le \sqrt{\frac{n}{3}}-\frac{1}{6}$. Therefore the desired conditions on $\tilde{y_{1}}$ and $\tilde{y_{2}}$ hold.

Thus if we show that some parallelogram is contained in a region for some large $n_0$ then it would be contained in the region for any $n>n_0$. For $n=6000$ one can show that there is a translation of the parallelogram spanned by $\left(\frac{2}{3},\frac{2}{3}\right)$ and $(2,1)$ centered on the line $y=x/2$ which is contained in the region between the curves $\tilde{y}_1$ and $\tilde{y}_2$. The geometry of this parallelogram was chosen in such a way as for any translation along the line $y=x/2$, the parallelogram must intersect at least one point of any of the possible lattice shifts (See Figure \ref{fig:helper11111} for an illustration of this). 

Note that by show above, we mean that we place the parralellogram visually inside the region using desmos, and know precisely the location of the four corners of the parralellogram in $(x,y)$. We can then calculate the points directly below, above, left, and right of these corners, and have containment.
\end{proof}

Next we show how one uses Theorem \ref{thm:alpha1} to prove Corollary \ref{cor:bootstrap}.
\begin{proof}[Proof of Corollary \ref{cor:bootstrap}]
We now show that if $n$ is large enough, and $m/n=\alpha\in[0.5,1]$, then the discrete double bubble perimeter for the volumes $m$, and $n$ is at most the ceiling of the continuous case double bubble perimeter plus two, which in this case would be $\left\lceil2\sqrt{6(n+m)}\right\rceil+2$.  

Let $t=\frac{n-m}{2}$.  Then the volumes $n-t$ and $m+t$ are the same.  For these two volumes, the continuous double bubble perimeter is $2\sqrt{12(n-t)}=2\sqrt{12(m+n)/2}=2\sqrt{6(n+m)}$, which is the same as the continuous double bubble perimeter for the volumes $m$ and $n$.  The argument above for $\alpha=1$ tells us that for $n$ large enough, there is a discrete configuration of the form in Figure \ref{fig:Figure9} with double bubble perimeter at most $\left\lceil2\sqrt{12(n-t)}\right\rceil+1=\left\lceil2\sqrt{12(n+m)/2}\right\rceil+1=\left\lceil2\sqrt{6(n+m)}\right\rceil+1$, and with both volumes at least $n-t=m+t$, i.e. the combined volume of the two sets is at least $2(n-t)=n+m$.  We can move the center line in such a way that one of the sets has volume at least $n$, and the other has volume at least $m$.  The problem at this point is that the vertical line $x'=L$ that passes through this center line may not be an integer, where $x'$ represents the horizontal axis.  The Figure \ref{fig:Figure10} demonstrates the situation:

\begin{figure}[H]
\begin{tikzpicture}[scale=0.6]

\draw[blue,  thick] (-4,0) to (4,0);
\draw[blue,  thick] (-4,0) to (-4,4);
\draw[blue,  thick] (-4,4) to (4,4);
\draw[blue,  thick] (4,4) to (4,0);

\draw[blue,  thick] (1,4) to (1,0);

\draw[blue,  dashed] (1,5) to (1,-1);
\draw (1,-1.3) node{$x'=L$};

\draw (-1.4,2) node{$\mu(A)\geq n$};
\draw (2.55,2) node{$\mu(B)\geq m$};

\draw (-1.6,-0.3) node{$x$};
\draw (2.4,-0.3) node{$y$};
\draw (-4.25,2) node{$z$};

\end{tikzpicture}
\caption{\label{fig:Figure10}}
\end{figure}

There are several ways of completing the construction so that the volumes are correct, the figures are both discrete shapes, and without increasing double bubble perimeter.  One way is to move the middle vertical line through which $x'=L$ passes to the right until it is positioned so that the vertical line $x'=\left\lceil L\right\rceil$.  If the volumes of the two figures are still at least $n$ and at least $m$, then we can reduce the volumes of both shapes via the same process given in Lemma \ref{lemma:Lem2.4} until they are both correct.  If, however, we have decreased the volume of $B$ in the above figure so that it is less than $m$, we can take volume from the bottom right corner of $A$ and add it to $B$ to create Figure \ref{fig:Figure11}:  

\begin{figure}[H]
\begin{tikzpicture}[scale=0.6]

\draw[blue,  thick] (-4,0) to (4,0);
\draw[blue,  thick] (-4,0) to (-4,4);
\draw[blue,  thick] (-4,4) to (4,4);
\draw[blue,  thick] (4,4) to (4,0);

\draw[blue,  thick] (1,4) to (1,1.5);
\draw[blue,  thick] (1,1.5) to (0.5,1.5);
\draw[blue,  thick] (0.5,1.5) to (0.5,0);

\draw[blue,  dashed] (1,5) to (1,-1);
\draw (1,-1.3) node{$x=\left\lceil L\right\rceil$};

\draw (-1.4,2) node{$\mu(A)\geq n$};
\draw (2.6,2) node{$\mu(B)=m$};

\draw (-2,-0.3) node{$x$};
\draw (2,-0.3) node{$y$};
\draw (-4.25,2) node{$z$};

\end{tikzpicture}
\caption{\label{fig:Figure11}}
\end{figure}

This procedure has added one to the double bubble perimeter.  However, since we started with a configuration with double bubble perimeter at most $\left\lceil2\sqrt{6(m+n)}\right\rceil+1$, we now have a configuration with double bubble perimeter at most $\left\lceil2\sqrt{6(m+n)}\right\rceil+2$.  Using Lemma \ref{lemma:Lem2.4} we can reduce the volume of $A$ to achieve $\mu(A)=n$ without increasing the double bubble perimeter.

\end{proof}


\appendix{}
\bibliographystyle{amsplain}
\bibliography{ri}

\providecommand{\bysame}{\leavevmode\hbox to3em{\hrulefill}\thinspace}
\providecommand{\MR}{\relax\ifhmode\unskip\space\fi MR }
\providecommand{\MRhref}[2]{%
  \href{http://www.ams.org/mathscinet-getitem?mr=#1}{#2}
}
\providecommand{\href}[2]{#2}
\begin{thebibliography}{10}

\bibitem{alexander1990wulff}
K.~Alexander, J.T. Chayes, and L.~Chayes, \emph{The wulff construction and
  asymptotics of the finite cluster distribution for two-dimensional bernoulli
  percolation}, Communications in mathematical physics \textbf{131} (1990),
  no.~1, 1--50.

\bibitem{alonso1996three}
L.~Alonso and R.~Cerf, \emph{The three dimensional polyominoes of minimal
  area}, The Electronic Journal of Combinatorics \textbf{3} (1996), no.~1, R27.

\bibitem{biskup2015isoperimetry}
M.~Biskup, O.~Louidor, E.~B. Procaccia, and R.~Rosenthal, \emph{Isoperimetry in
  two-dimensional percolation}, Communications on Pure and Applied Mathematics
  \textbf{68} (2015), no.~9, 1483--1531.

\bibitem{bodineau2000rigorous}
T.~Bodineau, D.~Ioffe, and Y.~Velenik, \emph{Rigorous probabilistic analysis of
  equilibrium crystal shapes}, Journal of Mathematical Physics \textbf{41}
  (2000), no.~3, 1033--1098.

\bibitem{cerf2006wulff}
R.~Cerf, \emph{The wulff crystal in ising and percolation models: Ecole
  d'et{\'e} de probabilit{\'e}s de saint-flour xxxiv-2004}, Springer, 2006.

\bibitem{duncan2020elementary}
P.~Duncan, R.~O'Dwyer, and E.~B. Procaccia, \emph{An elementary proof for the
  double bubble problem in $\ell^{1} 1$ norm}, arXiv preprint arXiv:2008.07767
  (2020).

\bibitem{foisy1993standard}
J.~Foisy, Manuel Alfaro~G., J.~Brock, N.~Hodges, and J.~Zimba, \emph{The
  standard double soap bubble in r2 uniquely minimizes perimeter}, Pacific
  journal of mathematics \textbf{159} (1993), no.~1, 47--59.

\bibitem{friedrich2021double}
M.~Friedrich, W.~G{\'o}rny, and U.~Stefanelli, \emph{The double-bubble problem
  on the square lattice}, arXiv preprint arXiv:2109.01697 (2021).

\bibitem{doublebubbleconj}
M.~Hutchings, F.~Morgan, M.~Ritor{\'e}, and A.~Ros, \emph{Proof of the double
  bubble conjecture}, Annals of Mathematics \textbf{155} (2002), no.~2,
  459--489.

\bibitem{milman2018gaussian}
E.~Milman and J.~Neeman, \emph{The gaussian double-bubble conjecture}, arXiv
  preprint arXiv:1801.09296 (2018).

\bibitem{morgan1998wulff}
F.~Morgan, C.~French, and S.~Greenleaf, \emph{Wulff clusters in $r^2$}, The
  Journal of Geometric Analysis \textbf{8} (1998), no.~1, 97.

\bibitem{procaccia2012concentration}
E.~B. Procaccia and R.~Rosenthal, \emph{Concentration estimates for the
  isoperimetric constant of the supercritical percolation cluster}, Electronic
  Communications in Probability \textbf{17} (2012), 1--11.

\bibitem{procaccia2016quenched}
E.~B. Procaccia, R.~Rosenthal, and A.~Sapozhnikov, \emph{Quenched invariance
  principle for simple random walk on clusters in correlated percolation
  models}, Probability theory and related fields \textbf{166} (2016), no.~3,
  619--657.

\end{thebibliography}

\end{document}